\documentclass[11pt]{amsart}

\usepackage[T1]{fontenc}
\usepackage[utf8]{inputenc}
\usepackage{amsmath,amssymb,amsthm,mathtools,mathrsfs}
\usepackage{graphicx}
\usepackage{hyperref}
\usepackage{caption}
\usepackage{subcaption}
\usepackage{float}
\usepackage{geometry}
\title{Random Knots via Stiefel manifolds}
\author{Alexander Kolpakov}
\author{Igor Rivin}
\date{\today}

\newtheorem{proposition}{Proposition}
\newtheorem{lemma}{Lemma}
\newtheorem{theorem}{Theorem}
\theoremstyle{remark}
\newtheorem{remark}{Remark}

\newcommand{\Ebb}{\mathbb{E}}
\newcommand{\Rbb}{\mathbb{R}}
\newcommand{\1}{\mathbf{1}}
\newcommand{\Kcal}{\mathcal{K}}
\newcommand{\row}{\operatorname{row}}
\newcommand{\Gr}{\operatorname{Gr}}
\newcommand{\St}{\operatorname{St}}

\begin{document}

\begin{abstract}
A fixed simplex, randomly projected into three dimensions and joined in
Hamiltonian order, produces a rich and unusually tractable model of random
stick knots.  We prove that Gaussian projections and Haar-random Stiefel
projections have exactly the same knot-type law, despite having different
metric shapes, and that at every stick budget the model gives positive
probability to precisely the knot types realizable with that many sticks.  Its
linear-algebraic structure yields an exact marginal distance law, an exact mean
planar-crossing count, and crossing concentration, while in the first
nontrivial six-stick case the complete tetrahedral sign pattern gives an exact
unknot-versus-handed-trefoil classifier for every generic sample.  The result
is a direct bridge from random projections and finite sign geometry to the
topology of random knots.
\end{abstract}

\maketitle

\section{Introduction and setup}

Assume \(N\ge4\).  Let \(e_1,\dots,e_N\) be the standard basis of \(\Rbb^N\), and set
\[
H:=\1^\perp\subset \Rbb^N,\qquad
\Pi_H:=I_N-\frac1N\1\1^T .
\]
For a matrix \(M\in\Rbb^{3\times N}\), write \(x_i:=Me_i\).  The Hamiltonian
cycle
\[
\mathcal C=(e_1,e_2,\dots,e_N,e_1)
\]
becomes the closed polygon
\[
\Gamma(M;\mathcal C)
:=
[x_1,x_2]\cup [x_2,x_3]\cup\cdots\cup[x_N,x_1]\subset\Rbb^3 .
\]
When this polygon is embedded, it is a polygonal, or stick, knot with \(N\) sticks.
Throughout, ambient isotopies start at the identity and hence preserve the orientation
of \(\Rbb^3\); a chiral knot and its mirror are distinct knot types unless we explicitly
merge mirror labels.  Knot types themselves are unoriented embedded subsets:
reversing the cyclic traversal changes no set \(\Gamma(M;\mathcal C)\) and hence
does not change its knot type.

All edge vectors \(e_i-e_{i+1}\) lie in \(H\).  Thus replacing \(M\) by
\(M\Pi_H\) only subtracts the centroid from the vertices and translates the polygon.  This is
the basic reason that the centered row space, rather than the absolute position of the
simplex vertices, controls the topological part of the construction.

We use two closely related models.
\begin{itemize}
\item[1.] Gaussian model: \(P\in\Rbb^{3\times N}\) has independent
entries distributed as \(\mathcal N(0,1/3)\), and the polygon is
\(\Gamma(P;\mathcal C)\).
\item[2.] Haar row-orthonormal model: \(Q\) is Haar distributed on
\[
\St_3(H):=\{Q\in\Rbb^{3\times N}:QQ^T=I_3,\ Q\1=0\},
\]
and the polygon is \(\Gamma(Q;\mathcal C)\).
\end{itemize}
The condition \(Q\1=0\) is the centering condition: if \(q_i\) is the \(i\)-th
column of \(Q\), then \(\sum_i q_i=0\).  Thus it is a linear dependence among
the columns, not among the three rows; \(QQ^T=I_3\) says the rows are
orthonormal, so \(Q|_H\) has rank three.
The first model is the direct Gaussian projection of the simplex.  The second retains only
the random three-dimensional shadow of the centered simplex and removes the singular-value
deformation of the Gaussian map.

\section{Relation to random-knot generation}

Many computational random-knot ensembles begin in \(\Rbb^3\): one samples lattice
self-avoiding polygons, equilateral or Gaussian polygons, confined polygons, or random-jump
polygons, and then studies the resulting knot statistics
\cite{SumnersWhittington,DiaoPippengerSumners,Diao,CantarellaDeguchiShonkwiler,EddyShonkwiler}.
Other ensembles are diagrammatic, such as Petaluma knots, grid diagrams, and random planar
diagrams \cite{EvenZoharModels,EvenZoharHassLinialNowikInvariants,
EvenZoharHassLinialNowikPetaluma,CantarellaChapmanMastin,ChapmanRandomDiagrams}.  The
projected-simplex ensemble differs from both families in a specific way.  The randomness
first chooses a three-dimensional linear shadow of a fixed high-dimensional simplex, and
only then reads off a polygon by a Hamiltonian ordering of the vertices.  In this respect the
closest geometric comparison is Westenberger's random-projection model, where a fixed curve
in high-dimensional Euclidean space is projected to random \(3\)-planes
\cite{WestenbergerRandomProjections}.  The computations below estimate Stiefel sector
probabilities by sampling, not by exact chamber-volume formulas.

Three features are relevant here.  First, the topological law has an exact Haar formulation
on \(\St_3(H)\).  Second, the Hamiltonian cycle is distributionally irrelevant: choosing
another cyclic ordering only relabels exchangeable simplex vertices.  Third, metric effects
are separated from topology: Gaussian projections add singular-value shear, while the
Haar row-orthonormal model removes it without changing knot type.

\section{Generic stick knots}

\begin{proposition}[Generic tameness]\label{prop:generic-tame}
If \(M\in\Rbb^{3\times N}\) has an absolutely continuous law, then
\(\Gamma(M;\mathcal C)\) is almost surely an embedded tame polygonal knot.  More precisely,
with probability one:
\[
x_i\ne x_j\ (i\ne j),\qquad
\text{no three distinct vertices are collinear,}
\]
and no two nonconsecutive edges intersect.
\end{proposition}

\begin{proof}
Coincidence of two vertices means \(M(e_i-e_j)=0\).  This is the vanishing of a
nonzero linear map of the entries of \(M\), hence a proper linear subspace of
\(\Rbb^{3N}\) and has Lebesgue measure zero.

For distinct \(i,j,k\), collinearity is equivalent to
\[
(x_i-x_k)\times (x_j-x_k)=0.
\]
The three coordinates of this cross product are polynomial functions of the entries of
\(M\).  They are not all identically zero: one may prescribe the three involved columns to
be \((0,0,0)\), \((1,0,0)\), and \((0,1,0)\).  Hence the collinearity locus is a
proper real-algebraic set and is null.

It remains to treat intersections of nonconsecutive edges.  Fix such a pair
\[
E_a=[x_a,x_{a+1}],\qquad E_b=[x_b,x_{b+1}].
\]
Let
\[
d_1=x_{a+1}-x_a,\qquad d_2=x_{b+1}-x_b.
\]
If the two segments intersect, then their supporting lines are coplanar, so
\[
\tau_{ab}(M):=\big(d_1\times d_2\big)\cdot(x_b-x_a)=0.
\]
This is a degree-three polynomial in the entries of \(M\).  It is not identically zero:
choose the four relevant columns to be
\[
x_a=(0,0,0),\quad x_{a+1}=(1,0,0),\quad
x_b=(0,1,0),\quad x_{b+1}=(0,0,1).
\]
Then \(d_1=(1,0,0)\), \(d_2=(0,-1,1)\), and
\[
\tau_{ab}=((1,0,0)\times(0,-1,1))\cdot(0,1,0)=-1.
\]
Thus \(\{\tau_{ab}=0\}\) is a proper algebraic hypersurface and is null.  Since there are
only finitely many vertex triples and edge pairs, the union of all bad events has probability
zero.  A polygonal embedded closed curve is tame, so the proposition follows.
\end{proof}

\section{Haar row factors and knot type}

The topological distribution of the Gaussian model is simpler than its metric distribution.
The singular values of the centered Gaussian map change lengths and angles, but they do not
change knot type.

\begin{proposition}[Gaussian polar decomposition]\label{prop:stiefel}
Let \(N\ge4\), and let \(P\in\Rbb^{3\times N}\) have independent centered Gaussian
entries with common variance \(\sigma^2>0\).  Set \(\widetilde P:=P\Pi_H\).  With probability
one, \(\widetilde P:H\to\Rbb^3\) has rank three, and its polar row factorization
\[
\widetilde P=AQ,\qquad
A:=(\widetilde P\widetilde P^T)^{1/2}\in GL^+(3,\Rbb),\qquad
Q\in\St_3(H),
\]
has \(Q\) Haar distributed on \(\St_3(H)\), independently of \(A\), and
\[
 \frac{A^2}{\sigma^2}\sim W_3(N-1,I_3).
\]
Here \(W_3(\nu,I_3)\) means the law of
\(\sum_{r=1}^{\nu}z_rz_r^T\) for independent \(z_r\sim\mathcal N_3(0,I_3)\).
The polygons
\(\Gamma(P;\mathcal C)\), \(\Gamma(\widetilde P;\mathcal C)\), and
\(\Gamma(Q;\mathcal C)\) have the same ambient isotopy class.
Consequently the knot type is independent of every measurable statistic of the
positive factor \(A\), including its unordered singular values and condition number.
\end{proposition}

\begin{proof}
Choose an isometry \(B:\Rbb^{N-1}\to H\) and write \(G=\widetilde P B\).  Then \(G\) is
a \(3\times(N-1)\) matrix of independent \(\mathcal N(0,\sigma^2)\) entries.  The standard
Gaussian polar decomposition \cite{Muirhead} gives
\[
 G=A U,\qquad A=(GG^T)^{1/2},\qquad U=A^{-1}G\in\St_3(\Rbb^{N-1}),
\]
where \(U\) is Haar, \(U\) is independent of \(A\), and
\(GG^T/\sigma^2\sim W_3(N-1,I_3)\).  One way to see the independence is to
disintegrate the Gaussian density, whose value depends only on \(GG^T\), under the
right action of \(O(N-1)\).  Since \(Q=UB^T\), the asserted Haar law and independence
follow.

Now \(A\) is positive definite and \(Q=A^{-1}\widetilde P\) has orthonormal rows in \(H\).
The centered vertices satisfy
\[
\widetilde Pe_i=A(Qe_i),\qquad i=1,\dots,N.
\]
Thus \(\Gamma(\widetilde P;\mathcal C)\) is the image of
\(\Gamma(Q;\mathcal C)\) under the orientation-preserving linear diffeomorphism
\(A:\Rbb^3\to\Rbb^3\).  Finally, \(\Gamma(P;\mathcal C)\) is a translate of
\(\Gamma(\widetilde P;\mathcal C)\).  Translations and orientation-preserving linear
diffeomorphisms preserve ambient isotopy class.
\end{proof}

\begin{remark}[Grassmannian shadow and chirality]
The row space \(\row(Q)\) is Haar distributed on \(\Gr(3,H)\).  A row space determines the
polygon only up to a left action of \(O(3)\): orientation-preserving choices give the same
chiral knot type, while an orientation-reversing choice gives the mirror.  Thus the full
chiral law is most naturally a Haar law on \(\St_3(H)\).  For metric observables and for
knot labels that identify mirror images, the law descends to the Grassmannian.
\end{remark}

\begin{proposition}[Distributional invariance across cycles]\label{prop:cycle-invariance}
Let \(\mathcal C\) and \(\mathcal C'\) be Hamiltonian cycles on \(\{1,\dots,N\}\).  Under
either the Gaussian model or the Haar row-orthonormal model, the corresponding random
polygons built from \(\mathcal C\) and \(\mathcal C'\) have the same distribution.
\end{proposition}

\begin{proof}
In the Gaussian model the columns are exchangeable, so relabeling the vertices by any
permutation does not change their joint distribution.  Since any Hamiltonian cycle can be
sent to any other by a permutation of labels, the polygon distributions agree.

In the Haar row-orthonormal model, a permutation matrix \(\Pi\in O(N)\) sends one cyclic
ordering to the other and preserves \(H=\1^\perp\).  Haar measure on \(\St_3(H)\) is
invariant under \(Q\mapsto Q\Pi\).  The same relabeling argument gives the result.
\end{proof}

\begin{proposition}[Mirror symmetry]\label{prop:mirror}
Under either model, the probabilities of a knot type \(K\) and its mirror
\(mK\) are equal.
\end{proposition}

\begin{proof}
Fix \(R\in O(3)\) with \(\det R=-1\).  Left multiplication \(Q\mapsto RQ\)
preserves Haar measure on \(\St_3(H)\), and the resulting polygon is the image
of the original polygon under an orientation-reversing linear isometry.
It therefore carries the mirror knot type.  The Gaussian law is likewise
invariant under \(P\mapsto RP\), or follows from
Proposition~\ref{prop:stiefel}.
\end{proof}

\begin{proposition}[Well-defined knot-type law]\label{prop:knot-type-law}
Under either model, the random stick knot has an almost surely well-defined ambient isotopy
class.  The resulting probability measure on the countable set of polygonal knot types is
Borel measurable and does not depend on the Hamiltonian cycle.
\end{proposition}

\begin{proof}
For the Gaussian model, Proposition~\ref{prop:generic-tame} gives an embedded tame polygon
almost surely.  Let \(Q\) be the polar row factor of the centered Gaussian matrix.  By
Proposition~\ref{prop:stiefel}, \(Q\) is Haar distributed on \(\St_3(H)\), and
\(\Gamma(Q;\mathcal C)\) is related to the centered Gaussian polygon by an invertible linear
map.  Thus the Haar row-orthonormal polygon is embedded almost surely as well.

Let \(\mathcal E\subset\St_3(H)\) be the set of frames giving embedded polygons.  The
preceding paragraph shows that \(\St_3(H)\setminus\mathcal E\) is Haar null.  On
\(\mathcal E\), the vertex coordinates depend continuously on the frame.  Standard
stability of PL embeddings implies that \(\mathcal E\) is open and that sufficiently
nearby polygonal embeddings are ambient isotopic \cite{EdwardsKirby}.  This formulation
also covers an embedded polygon having a deliberately inserted collinear subdivision
vertex; no noncollinearity margin is required.

Knot type is therefore locally constant on this open embedded-polygon set.  Thus the map
from \(\mathcal E\) to the set of knot
types is locally constant and hence Borel measurable.  Assigning any fixed knot type on the
null set \(\St_3(H)\setminus\mathcal E\) gives a measurable map on all of \(\St_3(H)\).
Proposition~\ref{prop:stiefel} transfers this law to the Gaussian model, and
Proposition~\ref{prop:cycle-invariance} gives independence of the Hamiltonian cycle.
\end{proof}

\section{Knot-type sectors and Stiefel probabilities}

Let \(\Delta_N\subset\St_3(H)\) be the discriminant of frames for which
\(\Gamma(Q;\mathcal C)\) is not an embedded polygon.  The preceding section shows that
\(\Delta_N\) is Haar null and that the knot-type map
\[
    \tau_N:\St_3(H)\setminus\Delta_N\longrightarrow \Kcal
\]
is locally constant, where \(\Kcal\) denotes the countable set of polygonal knot types.
Thus every connected component, or chamber, of
\(\St_3(H)\setminus\Delta_N\) carries a single knot type.  The converse need not hold:
many chambers can carry the same knot type.  The natural one-to-one object is the
knot-type sector
\[
    \mathcal S_N(K):=\tau_N^{-1}(K),
\]
which is generally a disconnected union of chambers.

Define the Stiefel knot probabilities by
\[
    p_N(K)
    :=
    \mu_{\mathrm{Haar}}\{Q\in\St_3(H):\Gamma(Q;\mathcal C)\text{ has knot type }K\}.
\]
Equivalently, after assigning any fixed value on the null set \(\Delta_N\),
\[
    p_N(K)=\mu_{\mathrm{Haar}}\bigl(\mathcal S_N(K)\bigr).
\]

\begin{proposition}[Support of the Stiefel knot law]\label{prop:stiefel-support}
For every knot type \(K\),
\[
    p_N(K)>0
    \quad\Longleftrightarrow\quad
    \operatorname{stick}(K)\le N.
\]
In particular, \(p_N(K)=0\) if \(\operatorname{stick}(K)>N\), and
\[
    \sum_{K\in\Kcal}p_N(K)=1.
\]
The support is finite.
\end{proposition}

\begin{proof}
If \(Q\in\St_3(H)\setminus\Delta_N\) represents \(K\), then its columns form an embedded
polygon with \(N\) sticks.  Hence \(\operatorname{stick}(K)\le N\), proving
\(p_N(K)=0\) for \(\operatorname{stick}(K)>N\).

Conversely, suppose \(\operatorname{stick}(K)\le N\).  Choose a polygonal representative
of \(K\) with at most \(N\) sticks and subdivide edges if necessary to obtain exactly \(N\)
labeled vertices.  After an arbitrarily small ambient isotopy, assume that the centered
coordinate rows span a three-dimensional subspace of \(H\).  Let
\(M\in\Rbb^{3\times N}\) be the centered vertex matrix, so \(M\1=0\) and \(MM^T\) is
positive definite.  Put
\[
    Q=(MM^T)^{-1/2}M .
\]
Then \(Q\in\St_3(H)\), and the original centered polygon is the image of the \(Q\)-polygon
under an orientation-preserving linear diffeomorphism of \(\Rbb^3\).  Thus \(Q\)
represents \(K\).

Embedded polygonal knots are stable under sufficiently small perturbations of their
vertices.  Therefore a nonempty open neighborhood of \(Q\) in \(\St_3(H)\) is contained in
\(\mathcal S_N(K)\).  Since Haar measure has full support on the compact smooth manifold
\(\St_3(H)\), this open set has positive measure, and \(p_N(K)>0\).

Finally, the map \(\tau_N\) is Borel measurable by Proposition~\ref{prop:knot-type-law},
the discriminant is null, and the countable sectors partition the embedded part of
\(\St_3(H)\).  Hence the probabilities sum to one.

For finiteness, take a generic planar projection of an \(N\)-stick polygon.
Only non--incident edge pairs can cross, so its crossing number is at most
\(N(N-3)/2\).  There are only finitely many knot diagrams with a bounded
number of crossings, and hence only finitely many knot types in the support.
\end{proof}

\begin{remark}[Stick-number data]
For tabulated exact values and bounds on \(\operatorname{stick}(K)\), we use
the KnotInfo stick-number table \cite{KnotInfoStick}; see also the recent
computational updates in
\cite{EddyShonkwiler,CantarellaRechnitzerSchumacherShonkwiler}.
\end{remark}

\begin{remark}[Topological sectors]
The discriminant decomposition is useful for language: knot invariants are constant on each
chamber, and \(p_N(K)\) is the Haar measure of the union of chambers carrying type \(K\).
For \(N\le 5\), stick number leaves only the unknot; at \(N=6\), the only possible
nontrivial type is the trefoil.
\end{remark}

\section{Metric normalizations}

The metric facts below fix the scale of the experiments.  The knot-type equivalence between
the Gaussian and Haar models does not require a distortion estimate.

\begin{lemma}[Raw Gaussian scale]\label{lem:raw-distance-law}
Let \(P\in\Rbb^{3\times N}\) have independent \(\mathcal N(0,1/3)\) entries.  For
every \(i\ne j\),
\[
\Ebb\|P(e_i-e_j)\|^2=2 .
\]
In particular, the Gaussian normalization preserves the squared simplex distance
\(\|e_i-e_j\|^2=2\) in expectation.
\end{lemma}

\begin{proof}
The \(i\)th and \(j\)th columns of \(P\) are independent centered Gaussian vectors with
covariance \((1/3)I_3\).  Their difference is centered Gaussian with covariance
\((2/3)I_3\).  The trace of this covariance matrix is \(2\), which is
\(\Ebb\|P(e_i-e_j)\|^2\).
\end{proof}

\begin{proposition}[Haar row-orthonormal scale]\label{prop:haar-scale}
Let \(Q\in\St_3(H)\).  If \(y_i:=Qe_i\), then \(\sum_i y_i=0\) and
\[
\frac{1}{\binom N2}\sum_{1\le i<j\le N}\|y_i-y_j\|^2
=\frac{6}{N-1}.
\]
Equivalently, the all-pair RMS distance is exactly
\[
\sqrt{\frac{2\cdot3}{N-1}}.
\]
\end{proposition}

\begin{proof}
The condition \(Q\1=0\) gives \(\sum_i y_i=0\).  The standard identity
\[
\sum_{i<j}\|y_i-y_j\|^2=N\sum_i\|y_i\|^2-\left\|\sum_i y_i\right\|^2
\]
therefore gives
\[
\sum_{i<j}\|y_i-y_j\|^2=N\,\operatorname{tr}(QQ^T)=3N.
\]
Dividing by \(\binom N2\) yields \(6/(N-1)\).
\end{proof}

\begin{proposition}[Exact marginal distance law]\label{prop:beta-distance}
Let \(Q\) be Haar distributed on \(\St_3(H)\), let \(i\ne j\), and assume
\(N\ge5\).  Then
\[
   Z_{ij}:=\frac{\|Q(e_i-e_j)\|^2}{2}
   \sim \operatorname{Beta}\left(\frac32,\frac{N-4}{2}\right).
\]
Consequently
\[
 \Ebb\|Q(e_i-e_j)\|^2=\frac6{N-1},\qquad
 \operatorname{Var}\!\left(\|Q(e_i-e_j)\|^2\right)
 =\frac{24(N-4)}{(N-1)^2(N+1)}.
\]
For \(N=4\), \(Z_{ij}=1\) deterministically.
\end{proposition}

\begin{proof}
The vector \(v=(e_i-e_j)/\sqrt2\) is a unit vector in the
\((N-1)\)-dimensional space \(H\).  The quantity \(\|Qv\|^2\) is the squared
length of the projection of \(v\) onto the Haar random three-plane
\(\row(Q)\).  Rotational invariance identifies its law with
\[
 \frac{X}{X+Y},\qquad X\sim\chi^2_3,\quad
 Y\sim\chi^2_{N-4},\quad X\ \hbox{and}\ Y\ \hbox{independent},
\]
which is the displayed beta distribution.  The formulas follow from the
standard beta moments.  If \(N=4\), the only three-plane in \(H\) is \(H\)
itself.
\end{proof}

\begin{remark}[The pair distances are not independent]
Proposition~\ref{prop:beta-distance} is a one-pair marginal statement.
Indeed Proposition~\ref{prop:haar-scale} gives the deterministic dependence
\[
   \sum_{i<j}\|Q(e_i-e_j)\|^2=3N.
\]
\end{remark}

\section{Exact crossings of the planar shadow}

Let \(C_N\) be the number of proper crossings between non--incident edges in
the planar polygon obtained from the first two coordinates of the
Gaussian model.  Equivalently, one may use the first two rows of a Haar
row-orthonormal projection: planar polar whitening is an invertible affine
map and therefore preserves every segment intersection.

\begin{lemma}[Four Gaussian points]\label{lem:four-gaussian-cross}
For four independent isotropic Gaussian points in \(\Rbb^2\), a prescribed
pair of disjoint segments crosses with probability
\[
    p_\times=\frac2\pi\arcsin(1/3).
\]
\end{lemma}

\begin{proof}
Augment the two coordinate rows of the four points by the all-ones row.
The kernel is almost surely a line in \(H_4=\1^\perp\), and that line is
uniform because the two centered Gaussian coordinate rows are independent
and isotropic in \(H_4\).  Choose a centered Gaussian representative
\(\lambda=(\lambda_1,\ldots,\lambda_4)\) of the line.  Its coordinates have
pairwise correlation \(-1/3\).

The four points are in convex position exactly when the signs of their
unique affine dependence split two against two.  The trivariate Gaussian
orthant formula
\[
 \Pr(Z_1,Z_2,Z_3>0)
 =\frac18+\frac1{4\pi}\sum_{a<b}\arcsin(\rho_{ab})
\]
shows that any prescribed one-against-three sign cone has probability
\[
 q=\frac18-\frac3{4\pi}\arcsin(1/3).
\]
There are eight such cones, so the convex-position probability is
\(1-8q=(6/\pi)\arcsin(1/3)\).  Conditional on convex position, exchangeability
makes each of the three pairings equally likely to be the diagonal pairing.
Dividing by three proves the claim.  This is the Gaussian instance of
Sylvester's four-point problem; see also \cite{FrickNewmanPegden}.
\end{proof}

\begin{theorem}[Exact mean and concentration]\label{thm:planar-crossings}
For every \(N\ge4\), under either the planar Gaussian or planar Haar model,
\[
 \Ebb C_N
 =\frac{N(N-3)}{\pi}\arcsin(1/3).
\]
Moreover, for every \(t>0\),
\[
 \Pr\!\left(\left|C_N-\Ebb C_N\right|\ge t\right)
 \le
 2\exp\!\left(-\frac{t^2}{2N(N-3)^2}\right).
\]
Equivalently, with \(L_N=N(N-3)/2\), for every \(\varepsilon\ge0\),
\[
 \Pr\!\left(\left|C_N/L_N-p_\times\right|\ge\varepsilon\right)
 \le 2e^{-\varepsilon^2N/8}.
\]
\end{theorem}

\begin{proof}
The cyclic \(N\)-gon has \(L_N=N(N-3)/2\) unordered pairs of non--incident
edges.  Each indicator has expectation \(p_\times\) by
Lemma~\ref{lem:four-gaussian-cross}, so linearity of expectation gives the
mean.  Changing one vertex can affect only indicators containing one of its
two incident edges, at most \(2(N-3)\) indicators.  McDiarmid's bounded
difference inequality \cite{McDiarmid} therefore gives
\[
 2\exp\!\left(
 -\frac{2t^2}{N[2(N-3)]^2}\right),
\]
which is the first tail bound; substituting \(t=\varepsilon L_N\) gives the
second.  Finally, the two-row polar decomposition couples the centered
Gaussian polygon to the Haar row-space representative by an invertible
linear map.  The first two rows of a Haar element of \(\St_3(H)\) have the
Haar law on \(\St_2(H)\), proving the Haar statement.
\end{proof}

\section{Computational experiments}

The companion \texttt{GaussianKnots} repository contains the experiments for the
three-dimensional model.  It samples the Hamiltonian cycle
\[
1\to2\to\cdots\to N\to1
\]
through projected simplex vertices, identifies the resulting stick knot with
\texttt{pyknotid} \cite{pyknotid}, and records both knot type and metric deformation.  Two
projection models
are compared:
\begin{itemize}
\item[1.] Gaussian model: sample \(P_{ab}\sim\mathcal N(0,1/3)\), center by
\(P\Pi_H\), and use the translated-equivalent vertices \(Pe_i\);
\item[2.] Haar row-space model: sample a Haar three-plane in \(H\) and choose
a row-orthonormal coordinate representative.
\end{itemize}
The implementation fixes the signs of the QR columns by a deterministic pivot
convention.  It therefore chooses a representative of a Haar row space rather
than a Haar frame on \(\St_3(H)\).  Every reported metric statistic and every
mirror-merged knot label is invariant under the resulting target \(O(3)\)
change, so the tables estimate the corresponding theoretical quantities.
In particular, for a chiral pair these tables estimate the merged probability
\(p_N(\{K,mK\})=p_N(K)+p_N(mK)\), not either handed probability separately.
For chiral labels one should instead use the positive-diagonal QR convention
(or apply an independent Haar element of \(O(3)\)).
By Proposition~\ref{prop:stiefel}, the two models have the same knot-type probabilities
\(p_N(K)\).  The Gaussian run also records the metric shear retained by \(P\Pi_H\);
this shear can affect diagram complexity and numerical conditioning, so ambiguous
classifications are reported separately.

For a fixed \(N\), the natural Monte Carlo estimator of the Stiefel probability is
\[
    \widehat p_N(K)
    =
    \frac{1}{M}\sum_{r=1}^M
    \mathbf 1\{\tau_N(Q_r)=K\},
\]
where \(Q_1,\dots,Q_M\) are independent Haar samples in \(\St_3(H)\).  We report the
sample size, unknown or ambiguous classifications, and Wilson binomial intervals for the
displayed rates.
If \(u\) of the \(M\) numerical classifications are unresolved and \(x\) are resolved
nontrivial knots, the data themselves only bracket the empirical nontrivial fraction by
\([x/M,(x+u)/M]\).  The displayed Wilson interval is for the recorded resolved-nontrivial
indicator; it is not a correction for classifier uncertainty.

\begin{figure}[ht]
\centering
\includegraphics[width=0.82\textwidth]{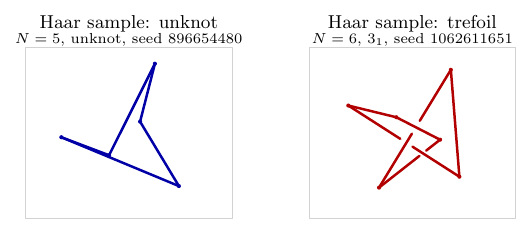}
\caption{Two fixed Haar row-orthonormal samples from the projected-simplex model, drawn as
generic planar diagrams with undercrossing gaps.  The left sample has \(N=5\) and is
unknotted by stick number.  The right sample has \(N=6\), the minimum possible stick number
for a trefoil, and was identified as \(3_1\).  The figure is reproducible from the
companion script.}
\label{fig:stiefel-examples}
\end{figure}

For each \(N=5,\dots,12\) we ran \(1000\) samples with seed \texttt{20260604}.  The knot
catalogue database was installed for \texttt{pyknotid}, and the Numba fast backend was
available.  Table~\ref{tab:knot-rates} gives the resulting nontrivial-knot rates.
The two Monte Carlo runs use matching per-sample seeds and are therefore coupled,
not independent.  They are still not expected to agree entrywise because their
generation maps differ;
Proposition~\ref{prop:stiefel} gives the same population knot-type law, with finite-sample
variation and classification ambiguity visible in the table.

\begin{table}[ht]
\centering
\scriptsize
\begin{tabular}{c|cc|cc}
\(N\) & Haar nontrivial & Haar 95\% CI & Gaussian nontrivial & Gaussian 95\% CI\\
\hline
5  & \(0/1000\)   & \(0.000\)--\(0.004\) & \(0/1000\)   & \(0.000\)--\(0.004\)\\
6  & \(2/1000\)   & \(0.001\)--\(0.007\) & \(5/1000\)   & \(0.002\)--\(0.012\)\\
7  & \(13/1000\)  & \(0.008\)--\(0.022\) & \(9/1000\)   & \(0.005\)--\(0.017\)\\
8  & \(26/1000\)  & \(0.018\)--\(0.038\) & \(39/1000\)  & \(0.029\)--\(0.053\)\\
9  & \(49/1000\)  & \(0.037\)--\(0.064\) & \(38/1000\)  & \(0.028\)--\(0.052\)\\
10 & \(66/1000\)  & \(0.052\)--\(0.083\) & \(84/1000\)  & \(0.068\)--\(0.103\)\\
11 & \(110/1000\) & \(0.092\)--\(0.131\) & \(104/1000\) & \(0.087\)--\(0.124\)\\
12 & \(133/1000\) & \(0.113\)--\(0.155\) & \(128/1000\) & \(0.109\)--\(0.150\)
\end{tabular}
\caption{Empirical resolved-nontrivial counts for the Haar row-space and Gaussian
models, with nominal Wilson 95\% binomial intervals for the displayed indicator.  The Haar run had
unknown or ambiguous classifications for \(1\) sample at \(N=10\), \(2\) samples at
\(N=11\), and \(1\) sample at \(N=12\).  The Gaussian run had unknown or ambiguous
classifications for \(1\) sample at each of \(N=9,10,11\), and \(11\) samples at \(N=12\).}
\label{tab:knot-rates}
\end{table}

\subsection{An exact sign classifier for six sticks}

For four distinct vertices define the alternating orientation sign
\[
 \chi(a,b,c,d)
 =
 \operatorname{sgn}\det(v_b-v_a,v_c-v_a,v_d-v_a).
\]
At a generic configuration the \(15=\binom64\) signs are all nonzero.
For five distinct vertices \(a,b,c,p,q\), the signed cofactors
\[
\begin{split}
(&\chi(b,c,p,q),-\chi(a,c,p,q),\chi(a,b,p,q),\\
 &-\chi(a,b,c,q),\chi(a,b,c,p))
\end{split}
\]
are the signs of the coefficients in their unique affine dependence.
It follows that the open segment \(pq\) pierces the relative interior of
the triangle \(abc\) exactly when the first three signs agree, the last
two agree, and those two common signs are opposite.  When this happens,
the oriented piercing contribution is \(\chi(a,b,c,q)\).  This is simply
the cofactor identity for the augmented \(4\times5\) vertex matrix, so the
test uses only the chirotope.

Let \(I(a,b,c;p,q)\in\{-1,0,1\}\) denote that oriented contribution, or
zero when there is no piercing, and put
\[
\begin{aligned}
 \Delta_2&=I(1,2,3;4,5)+I(1,2,3;5,6),\\
 \Delta_4&=I(3,4,5;6,1)+I(3,4,5;1,2),\\
 \Delta_6&=I(5,6,1;2,3)+I(5,6,1;3,4).
\end{aligned}
\]
Calvo proves for an embedded generic hexagon that all three values are
\(+1\) precisely for a right trefoil, all are \(-1\) precisely for a left
trefoil, and at least one is zero precisely for the unknot
\cite{Calvo}.  Thus the full \(15\)-sign vector determines the topological
knot type; it also determines Calvo's geometric curl through
\(\chi(1,2,3,5)\).  No connectedness assumption about an order-type
realization set is involved.

We applied this exact finite sign rule to the separate \(500{,}000\)-sample
\(N=6\) run, after grouping signatures by Hamiltonian \(D_6\)-symmetry and
global sign when only the mirror-merged label was required.  The rule found
one canonical trefoil signature,
\[
 \texttt{+-++-++--+-++-+},
\]
and hence
\[
    \widehat p_6(3_1)=1856/500000=0.003712,
    \qquad
    95\%\ {\rm CI}=0.003547\text{--}0.003884.
\]
The classifier also agreed on every one of the saved \(1000\) Haar and
\(1000\) Gaussian \(N=6\) samples with the separately computed \texttt{pyknotid}
label.  The rule is an exact classifier of a generic sign vector; the
recorded signs and the Monte Carlo frequency are, of course, numerical.
The Wilson intervals for the direct Haar and Gaussian counts in
Table~\ref{tab:knot-rates}, \(2/1000\) and \(5/1000\), are much wider and both overlap
\(0.003712\).  The direct runs are therefore consistent with the larger sign-classified run but too
small to sharpen the \(N=6\) estimate.  In particular, Calvo's theorem upgrades
the bucket labeling, not the Monte Carlo estimate into an exact chamber-volume formula.

\begin{table}[ht]
\centering
\scriptsize
\begin{tabular}{c|rrrrrrrrrr}
\(N\) & \(0_1\) & \(3_1\) & \(4_1\) & \(5_1\) & \(5_2\) & \(6_1\) & \(6_2\) & \(6_3\)
& other & unknown\\
\hline
5  & 1000 & 0   & 0  & 0 & 0 & 0 & 0 & 0 & 0 & 0\\
6  & 998  & 2   & 0  & 0 & 0 & 0 & 0 & 0 & 0 & 0\\
7  & 987  & 13  & 0  & 0 & 0 & 0 & 0 & 0 & 0 & 0\\
8  & 974  & 23  & 3  & 0 & 0 & 0 & 0 & 0 & 0 & 0\\
9  & 951  & 44  & 3  & 0 & 2 & 0 & 0 & 0 & 0 & 0\\
10 & 933  & 57  & 5  & 2 & 1 & 0 & 0 & 1 & 0 & 1\\
11 & 888  & 81  & 12 & 6 & 8 & 1 & 1 & 1 & 0 & 2\\
12 & 866  & 101 & 18 & 5 & 4 & 0 & 0 & 1 & 4 & 1
\end{tabular}
\caption{Haar row-orthonormal knot-type counts from the same \(1000\)-sample runs.  The
``other'' column records identified nontrivial labels outside the displayed
low-crossing types; the ``unknown'' column records ambiguous catalogue candidate lists or
unresolved classifications.}
\label{tab:haar-type-counts}
\end{table}

These data are not directly comparable to random-polygon or random-diagram ensembles.
Equilateral, Gaussian, confined, and quaternionic polygon models sample curves directly in
\(\Rbb^3\); grid and planar-diagram models start from planar combinatorics.  The Stiefel
model samples normalized linear shadows of a fixed simplex, and its planar diagrams inherit
crossing signs from a height coordinate.  Thus the induced shadow and crossing-sign laws
differ from the diagrammatic models cited above.

The support theorem gives a deterministic benchmark for finite sampling: every knot type
with \(\operatorname{stick}(K)\le N\) has \(p_N(K)>0\), while failure to observe it at
sample size \(M\) only says that its sector was not hit.

The runs also recorded distance deformation.  For all \(\binom N2\) vertex pairs, let
\[
D_{\rm pair}:=\frac{\max_{i<j}\|x_i-x_j\|}{\min_{i<j}\|x_i-x_j\|}.
\]
The original simplex distance is \(\sqrt2\).  By Lemma~\ref{lem:raw-distance-law}, the Gaussian normalization has expected squared pair distance \(2\).  By
Proposition~\ref{prop:haar-scale}, the row-orthonormal Haar model has all-pair RMS distance
\(\sqrt{2\cdot 3/(N-1)}\), so it requires the global rescaling factor
\(\sqrt{(N-1)/3}\) for metric comparison.  Table~\ref{tab:metric-deformation} reports mean
all-pair max/min distortion, mean RMS distance divided by \(\sqrt2\), and the mean
scale-free all-pair minimum and maximum after per-sample RMS normalization.

\begin{table}[ht]
\centering
\scriptsize
\begin{tabular}{c|ccc|ccc}
& \multicolumn{3}{c|}{Haar} & \multicolumn{3}{c}{Gaussian}\\
\(N\) & \(D_{\rm pair}\) & RMS\(/\sqrt2\) & norm.\ range
& \(D_{\rm pair}\) & RMS\(/\sqrt2\) & norm.\ range\\
\hline
5  & 2.448 & 0.866 & \(0.531\)--\(1.151\) & 4.335  & 0.982 & \(0.411\)--\(1.492\)\\
6  & 3.467 & 0.775 & \(0.422\)--\(1.267\) & 5.477  & 0.990 & \(0.342\)--\(1.578\)\\
7  & 4.524 & 0.707 & \(0.362\)--\(1.360\) & 6.270  & 0.988 & \(0.311\)--\(1.635\)\\
8  & 5.501 & 0.655 & \(0.319\)--\(1.436\) & 7.244  & 0.988 & \(0.282\)--\(1.693\)\\
9  & 6.384 & 0.612 & \(0.286\)--\(1.500\) & 8.464  & 0.978 & \(0.256\)--\(1.740\)\\
10 & 7.157 & 0.577 & \(0.260\)--\(1.552\) & 9.174  & 0.988 & \(0.236\)--\(1.781\)\\
11 & 8.031 & 0.548 & \(0.238\)--\(1.609\) & 9.958  & 0.997 & \(0.219\)--\(1.813\)\\
12 & 8.632 & 0.522 & \(0.230\)--\(1.648\) & 10.895 & 0.991 & \(0.206\)--\(1.853\)
\end{tabular}
\caption{Mean all-pair distance deformation over the same \(1000\) samples per \(N\).  The
``norm.\ range'' columns remove one global scale by dividing all projected pair distances in
a sample by their RMS.  In these runs the normalized ranges show moderate shear and no
near-collapse of the vertex set.}
\label{tab:metric-deformation}
\end{table}

\section{Formal verification}

The Lean~4/mathlib development at
\url{https://github.com/sashakolpakov/GaussianKnots} defines the Gaussian
product measure, proper-crossing statistic, vertex replacement, and
six-vertex determinant signature, and verifies the metric identities,
Gaussian centering invariance, \(2(N-3)\) sensitivity, and circuit/intersection
criterion.  From the complete four-point crossing-probability lemma as a typed
input, Lean derives the stated mean; from a typed Doob certificate it derives
both tail bounds by mathlib's Azuma--Hoeffding theorem; and it computes all six
Calvo records from one 15-determinant signature.  The four-point lemma and its
orthant argument, construction of the Doob certificate, polar/Wishart transfer,
Haar beta law, PL isotopy, and Calvo's topological classification remain
external.  The repository contains the declaration-level audit and build
instructions.

\section{Reproducibility}

The knot-type and distance-deformation experiments are maintained in the same
companion repository.
The corresponding Sphinx documentation is published at
\begin{center}
\url{https://sashakolpakov.github.io/GaussianKnots/}.
\end{center}
The main commands used for the tables and the \(N=6\) sign-classified run are
\begin{verbatim}
python3 scripts/run_knot_experiment.py \
  --projection-model haar \
  --vertices 5,6,7,8,9,10,11,12 \
  --samples 1000 \
  --seed 20260604 \
  --output-dir results/haar_N5-12_1000

python3 scripts/run_knot_experiment.py \
  --projection-model gaussian \
  --vertices 5,6,7,8,9,10,11,12 \
  --samples 1000 \
  --seed 20260604 \
  --output-dir results/gaussian_N5-12_1000

python3 scripts/repro/order_type_grouped_volume.py \
  --vertices 6,7,8 \
  --samples 500000 \
  --seed 20260604 \
  --classify-top-groups 200 \
  --checks-per-group 3 \
  --direct-classify-samples 1000 \
  --output-dir results/order_type_grouped_volume_N6-8_500k
\end{verbatim}
The exact-law checks and the Calvo reclassification are reproduced from this
note's repository by
\begin{verbatim}
python3 scripts/repro/validate_exact_projection_laws.py \
  --crossing-samples 10000 \
  --beta-samples 50000 \
  --batch-size 500

GaussianKnots/.venv/bin/python scripts/repro/calvo_n6_exact.py \
  --expected-trefoils 1856 \
  --validate-standard-direct \
  --audit-signs
\end{verbatim}
The optional stringent audit replays the complete seeded random-number stream,
including the reservoir-sampling draws, and recovers all \(155\) saved
six-vertex sign buckets with no count discrepancy.  The closest sampled
tetrahedral wall has scale-normalized Pl\"ucker margin
\(3.2516397907\cdot10^{-10}\); its affine determinant has magnitude
\(7.9648583145\cdot10^{-10}\), about \(7.96\) times the sign cutoff.
Recomputing the \(100\) closest cases with \(100\)-digit decimal arithmetic
produces no sign mismatch and no canonical-signature mismatch.  This is a
numerical robustness certificate for the finite sample, not a symbolic
evaluation of its Haar probability.
Those runs use the performance extra of the \texttt{pyknotid} fork, the installed catalogue
database, and the Numba fast backend when available.

\section*{Acknowledgments}

The authors used OpenAI's GPT-5.6-sol as a research and coding assistant to
help strengthen the results, improve the exposition, and generate and audit
the Python and Lean~4 code accompanying this paper.  The authors reviewed the
resulting mathematics and code and take responsibility for the final content.

\end{document}